\documentclass[a4paper,12pt,reqno]{amsart}
\usepackage{amsmath,amsfonts,amssymb}
\usepackage{tikz}
\usepackage{multirow}
\usepackage{graphicx}
\usepackage[all]{xy}
\usepackage{verbatim}
\usepackage{url}
\usepackage{hyperref}
\usepackage{longtable}
\usepackage{float}
\floatstyle{plaintop}
\restylefloat{table}
\calclayout
\theoremstyle{plain}
\newtheorem{theorem}{Theorem}[section]
\newtheorem{lemma}[theorem]{Lemma}

\newtheorem{proposition}[theorem]{Proposition}

\theoremstyle{definition}
\newtheorem*{definition}{Definition}

\theoremstyle{remark}

\newcommand{\C}{\mathbb{C}}

\newcommand{\Z}{\mathbb{Z}}
\newcommand{\T}{\mathcal{T}}
\newcommand{\z}{\mathbf{z}}

\newcommand{\Hy}{\mathbb{H}}
\newcommand{\Oc}{\mathcal{O}}
\newcommand{\Sc}{\mathcal{S}}
\newcommand{\B}{\mathcal{B}}
\newcommand{\Pc}{\mathcal{P}}

\def\Zbb{\mathbb{Z}}

\begin{document}

\newpage
\title{On the volume  and the Chern-Simons invariant for the hyperbolic alternating knot orbifolds}
\author{Ji-Young Ham and Joongul Lee}
\address{Da Vinci College of General Education, Chung-Ang University, 
General Education Building, 84 HeukSeok-Ro, DongJak-Gu, Seoul,
06974,\\
  Korea} 
\email{jiyoungham1@gmail.com.}

\address{Department of Mathematics Education, Hongik University, 
94 Wausan-ro, Mapo-gu, Seoul,
04066\\
   Korea} 
\email{jglee@hongik.ac.kr}



\maketitle 

\markboth{Ji-Young Ham and Joongul Lee} 
{On the volume  and the Chern-Simons invariant for the alternating knot orbifolds}

\begin{abstract}
We extend the Neumann's methods in~\cite{N2} and give the explicit formulae for the volume and the Chern-Simons invariant for hyperbolic alternating knot orbifolds. 
\end{abstract}
\maketitle

\section{Introduction}
 
 We extend the Neumann's methods in~\cite{N2} using Zickert's methods in~\cite{Z1,Z16,GZ17} and the formula for the complex volume of a hyperbolic knot in~\cite{CKK1} to present explicit formulae for the  volume and the Chern-Simons invariant of the hyperbolic alternating knot orbifolds.

In~\cite{N2}, Neumann defined the extended Bloch group, $\widehat{\B} (\C)$, and showed that this group is isomorphic to $H_3(\text{PSL}(2,\C);\Z)$ 
by lifting the Bloch-Wigner map 
$$H_3(\text{PSL}(2,\C);\Z) \rightarrow \B(\C)$$
 to an isomorphism
\begin{align*}
\lambda:\, \xymatrix{H_3(\text{PSL}(2,\C);\Z) \ar[r]^-{\cong} & \widehat{\B} (\C)}.
\end{align*}
On $\widehat{\B} (\C)$ he defined an analytic function (an extended version of Roger's dilogarithm function);
$$R: \, \widehat{\B} (\C) \rightarrow \C/\pi^2 \Z$$
and showed that the composition
$$R \circ \lambda: \, H_3(\text{PSL}(2,\C);\Z) \rightarrow \C/\pi^2 \Z$$
is the Cheeger-Chern-Simons class, which  also can be written as $i(\text{vol}+i \text{cs})$, where $cs$ is the universal Chern-Simons class.
Neumann showed that any complete hyperbolic 3-manifold $M$ of finite volume has a natural “fundamental class” in $H_3(\text{PSL}(2,\C);\Z)$. 
Hence this fundamental class of $M$ in $H_3(\text{PSL}(2,\C);\Z)$ determines an element $\widehat{\beta}(M) \in \widehat{\B} (\C)$. 
Neumann describe $\widehat{\beta}(M)$ directly in terms of an ideal triangulation $\T$ of $M$. By substituting truncated simplices for simplices, one can define $\overline{\T}$ of $M$ from $\T$ of $M$.
In~\cite{Z1,Z16,GZ17}, Zickert introduced a way of using a $P$-decoration on $\widetilde{\overline{\T}}$  (liftings of $\overline{\T}$ to the universal cover of $M$) to compute complex volume, where $P$ is the subgroup of upper triangular matrices with $1$'s on the diagonal. Let $\rho$ be a geometric representation. A $P$-decoration of $\rho$ is 
a $\rho$-equivariant assignment of $P$-cosets to each triangular face of $\widetilde{\overline{\T}}$. $P$-decorations
are in one-to-one correspondence with Fattened natural cocycles on $\overline{\T}$. A $P$-decoration of $\rho$ involves a choice of fundamental rectangle for each boundary component such that the triangulation induced on each boundary torus obtained by identifying the sides of the chosen fundamental rectangle agrees with the triangulation of the boundary torus induced by $\T$.
For the computation, we use the triangulation used in~\cite{CKK1}. The triangulation used in~\cite{CKK1} is dipicted on the chosen fundamental rectangle. Therefore we identify the sides of the fundamental rectangle by the identity. In Section~\ref{sec:shape}, in particular, we show that the shape of a $J(2n,-2m)$ knot orbifold can be obtained by a root of a certain polynomial and this polynomial, in fact, coincides with the Riley-Mednykh polynomial. 
In Section~\ref{sec:cvolume}, we present some explicit complex volume of $J(2n,-2m)$ knot orbifolds.

Some volume formulae for hyperbolic cone-manifolds and some formulae for Chern-Simons invariants for hyperbolic orbifolds of knots and links based on Schl\"{a}fli formula can be found in~\cite{HLM1,K1,K2,MR1,MV1,M1,DMM1,HLMR1,HLM2,HMP,HL1,HLMR,Tran1, Tran2,HLM2,HL,HL2,HLMR,A,A1}.
Some references for cone-manifolds are~\cite{CHK,T1,K1,P2,HLM1,PW,HMP}.

\section{The hyperbolic structure of alternating knot orbifolds}

Let $K$ be a hyperbolic alternating knot. The $r$-fold cyclic covering, $M$, of $S^3 \backslash K$ is the covering space corresponding to an index $r$ subgroup of the fundamental group of $\pi_1(S^3 \backslash K)$. An index $r$ subgroup can be identified with the kernel of the composite of the following two maps:
$$\pi_1(S^3 \backslash K) \xrightarrow{\phi} H_1(S^3 \backslash K) \cong \Z \xrightarrow{p} \Z_r.$$ From~\cite{Fox}, topologically, the $r$-fold cyclic covering branched along $K$ is the completion of $M$.
Naturally, a topological ideal triangulation of $S^3 \backslash K$ gives the topological ideal triangulation of a sheet of M and $r$ times of it will give the topological ideal triangulation of $M$. A solution to the gluing and completeness equations of~\cite{NZ} will give the hyperbolic structure to $M$. Since the gluing equations will be the same for each sheet and the product of moduli corresponding to the meridian is the $r$th power of the product of moduli corresponding to the meridian restricted to a sheet, a solution can be found by restricting the equations to a sheet (for explicit computations, see Section~\ref{sec:proof}), which will also give the hyperbolic structure to a sheet and to the completion of a sheet, the orbifold 
$\Oc(K,r)$ of $K$ with cone-angle $2 \pi/r$.

But the existence of the solution is not always guaranteed. Hence we will fix an ideal triangulation which will guarantee it. We assume the planar projection of an alternating knot is always reduced. For the topological ideal triangulation of $S^3 \backslash K$, we take the triangulation of $S^3\backslash \{K \cup \text{two points} \}$ described by~\cite{CKK1,GMT}, the \emph{octahedral $4$-term triangulation}, and to give the geometric structure to $M$ we send two points to $\partial \overline{\Hy}^3$ equivariantly, making sure that the resulting ideal simplices are non-degenerate (all four vertices distinct) as in~\cite[p.465]{N2}. With this ideal triangulation, a solution to the gluing and completeness equations can always be obtained to recover the hyperbolic structure of 
$S^3 \backslash K$~\cite{GMT,SY}. Hence the deformed solution can always be obtained to give the hyperbolic structure to a sheet of $M$ for sufficiently large $n$ in the sense of~\cite[Chap.5]{T1} and~\cite{K1}, therefore to $M$ and 
$\Oc(K,r)$~\cite[Sec.3]{DG1}.

\section{extended Bloch group and an extended version of Rodger's dilogarithm function}
The general reference for this section is ~\cite{N2, Z1,DZ,GZ}.

\begin{definition}
The \emph{pre-Bloch group} $\Pc(\C)$ is an abelian group generated by symbols $[z]$, $z \in \C \backslash \{0, 1\}$, subject to the relation
$$[x]-[y]+ \left[\frac{y}{x} \right]-\left[\frac{1-x^{-1}}{1-y^{-1}} \right] +\left[\frac{1-x}{1-y} \right] =0. $$
This relation is called the \emph{five-term relation}.
\end{definition}
An element in pre-Bloch group can be considered as a cross-ratio of a congruence class of an ideal simplex. We consider the vertex ordering as part of the data defining an ideal simplex.

\begin{definition}
The Bloch group $\B(\C)$ is the kernel of the homomorphism
$$\nu : \, \Pc ( \C ) \rightarrow \wedge^{2}_{\Z} ( \C^{*}) $$
defined by mapping a generator $[z]$ to $z \wedge (1 - z)$.
\end{definition}
The image of the map $\nu$ is  the complex version of the \emph{Dehn invariant} and make the $ker(\nu)$ into the orientation sensitive
\emph{scissors congruence group} containing scissors congruence classes of complete hyperbolic 3-manifolds of finite volume.

\begin{definition}
Let $\Delta$ be an ideal simplex with cross-ratio $z$. A (combinatorial) flattening of $\Delta$ is a triple of complex numbers of the form
\begin{align*}
(w_0, w_1, w_2) =  (\log{z} + p \pi i, - \log{(1- z)} + q \pi i,
-\log{z} + \log{(1- z)} - p \pi i-q \pi i)
\end{align*}
with $p, q \in \Z$. We call $w_0,\ w_1$ and $w_2$ log parameters.
\end{definition}
We always use the principal branch having imaginary part in the interval $(-\pi, \pi ]$. 
Since log parameters uniquely determine $z$, we can write a flattening as $(z; p, q)$ \cite[Lemma 3.2]{N2}.

\begin{definition}
Let $z_0, \ldots , z_4$ be five distinct points in $\C \cup \{\infty\}$, and let $\Delta_ i$ denote the simplices 
$[z_0, \ldots ,\hat{z}_i, \ldots,z_4]$. Suppose that 
$(w_0^i ,w_1^i ,w_2^i)$ are flattenings of the simplices  $\Delta_i$. Every edge $z_iz_j$ belongs to exactly three of the  $\Delta_i$'s and therefore has three associated log parameters. The flattenings are said to satisfy the flattening condition if for each edge the signed sum of the three associated log parameters is zero. The sign is positive
if and only if $i$ is even.
\end{definition}

It follows directly from the definition that the flattening condition is equivalent to the
following ten equations:
\begin{alignat*}{2}
& z_0z_1  :  w_0^2 -w_0^3 +w_0^4 =0, \qquad && z_0z_2  : -w_0^1 -w_2^3 +w_2^4 =0, \\
& z_1z_2  :  w_0^0 -w_1^3 +w_1^4 =0,  && z_1z_3  : w_2^0 +w_1^2 +w_2^4 =0, \\
& z_2z_3  :  w_1^0 -w_1^1 +w_0^4 =0,  && z_2z_4  : w_2^0 -w_2^1 -w_0^3 =0, \\
& z_3z_4  :  w_0^0 -w_0^1 +w_0^2 =0,  && z_3z_0  : -w_2^1 +w_2^2 +w_1^4 =0, \\
& z_4z_0  : -w_1^1 +w_1^2 -w_1^3 =0,  && z_4z_1  : w_1^0 +w_2^2 -w_2^3 =0.
\end{alignat*}

\begin{definition}
The \emph{extended pre-Bloch group} $\widehat{\Pc}(\C)$ is the free abelian group on flattened ideal simplices subject to the relations
\begin{enumerate}
\item   $\sum^4_{i = 0} ( -1 )^ i ( w_0^i , w_1^i , w_2^i ) = 0$ if the flattenings satisfy the flattening condition, \\
\item  $(z;p,q)+(z;p',q') = (z;p,q')+(z;p',q)$.
\end{enumerate}
\end{definition}
The first relation lifts the relation the five term relation. It is therefore called the \emph{lifted five-term relation}. The second relation is called the \emph{transfer relation}. We shall denote the class of $(z; p, q)$ in $\widehat{\Pc}(\C)$ by $[z; p, q]$.

\begin{definition}
The extended Bloch group $\widehat{\B}(\C)$ is the kernel of the homomorphism
$$\nu' : \widehat{\Pc} ( \C ) \rightarrow \wedge^ 2_{\Z}(\ C )$$ 
defined on generators by $(w_0, w_1, w_2)  \mapsto w_0 \wedge w_1$.
\end{definition}

The following Theorem~\ref{thm:chi} describe the relationship of the extended groups with the “classical” ones.
\begin{theorem}~\cite[Theorem 7.5]{N2} \label{thm:chi}
There is a commutative diagram with exact rows and columns
$$
\xymatrix{& 0 \ar[d] & 0 \ar[d] & 0 \ar[d]\\
0 \ar[r] & \mu^* \ar[d]_{\chi | \mu^*} \ar[r] & \C^* \ar[d]_{\chi} \ar[r] & \C^*/\mu^*  \ar[d]_{\beta} \ar[r] & 0 \\
0  \ar[r] &  \widehat{\B} (\C) \ar[d] \ar[r] & \widehat{\Pc} (\C) \ar[d] \ar[r]^{\nu} & \C \wedge \C  \ar[d]_{\epsilon} \ar[r] & K_2(\C) \ar[d]_{=} \ar[r]  & 0 \\
0 \ar[r] & \B(\C) \ar[d] \ar[r] & \Pc(C) \ar[d] \ar[r]^{\nu'} & \C^* \wedge \C^*  \ar[d] \ar[r] & K_2(\C) \ar[d] \ar[r]  & 0  \\
& 0 & 0 & 0 & 0}
$$
Here $\mu^{*}$ is the group of roots of unity and the labeled maps defined as follows:
\begin{align*}
\chi ( z ) & = [ z ; 0 , 1 ] - [ z ; 0 , 0 ] \in \widehat{\Pc} ( \C ) ; \\
\nu [z; p, q] & = (\log{z} + p \pi i) \wedge (- \log{(1 - z)} + q \pi i); \\
\nu'[z]  & = 2 (z \wedge (1 - z)) ; \\
\beta [z] & = \log{z} \wedge \pi i; \\
\epsilon(w_1 \wedge w_2) &=-2(e^{w_1} \wedge e^{w_2});
\end{align*}
and the unlabeled maps are the obvious ones.
\end{theorem}

The following definition gives an extended version of Rodger's dilogarithm function which is defined by Neumann.
\begin{definition}
Define
\begin{align*}
R(z; p, q) & = \mathcal{R}(z) + \frac{\pi i}{2} (p \log{(1 - z)} + q \log{z}) -\frac{\pi^2}{6}\\
& =\textnormal{Li}_2(z)-\frac{\pi^2}{6}+ \frac{1}{2} q \pi i \log{z}+\frac{1}{2} \log{(1-z)} (\log{z} +p \pi i )
\end{align*}
where $\mathcal{R}$ is the Rogers dilogarithm function
$$\mathcal{R}(z)= \frac{1}{2} \log{z} \log{(1-z)}-\int_0^ z \frac{\log{(1-t)}}{t}dt$$
and $\textnormal{Li}_2$ is the dilogarithm function
$$\textnormal{Li}_2(z)= -\int_0^ z \frac{\log{(1-t)}}{t}dt.$$
\end{definition}
Then $R$ gives a homomorphism $R:\ \widehat{\Pc} (\C) \rightarrow \C/ \pi^2 \Z$~\cite[Proposition 2.5]{N2}.
\section{Octahedral $4$-term triangulation and the potential function}
\subsection{Octahedral $4$-term triangulation}
Among the references for this subsection, we use Section 3 of~\cite{Weeks} and Section 3 of~\cite{CKK1}. Denote the tubular neighborhood of an alternating knot $K$, whose planar projection on the equatorial $S^2$ of $S^3$ is reduced, by 
$S^2 \times I$. Now, triangulate $(S^2 \times I) \backslash K$ instead of $S^3 \backslash K$ with two inside points as vertices. To triangulate $(S^2 \times I) \backslash K$, cut straight down through it just as you would cut cookie dough with a cookie cutter. Then up to homotopy for each crossing of $K$, we have an octahedron as in the right side of Figure~\ref{fig:order}~\cite{Weeks,KKY}. We assign vertex orderings of the tetrahedra in the right side of Figure~\ref{fig:order} by assigning $0$ to $E_k$, $1$ to $F_k$, $2$ to $A_k$ and $C_k$, and $3$ to $B_k$ and $D_k$. Then, when two edges are glued together in the triangulation, the orientations of the two edges induced by each vertex orderings coincide. This was called \emph{edge-orientation consistency} in~\cite{CKK1}. This property is required to use the formulae in~\cite{N2,Z1}. Note that right before using a cookie cutter, $A_k$ and $C_k$ (resp. $B_k$ and $D_k$) were one point. After some deformation up to homotopy, they became two points. Hence they have the same vertex ordering. We assign formal variables $z_a,\ z_b,\ z_c,\ z_d$ to $A_k,\ B_k,\ C_k,\ D_k$, in order, as in the left side of Figure~\ref{fig:order}. We also assign a formal shape parameter to each tetrahedron as in Figure~\ref{fig:shapeParameter}.
\begin{figure} 
\begin{center}
\resizebox{10cm}{!}{\includegraphics{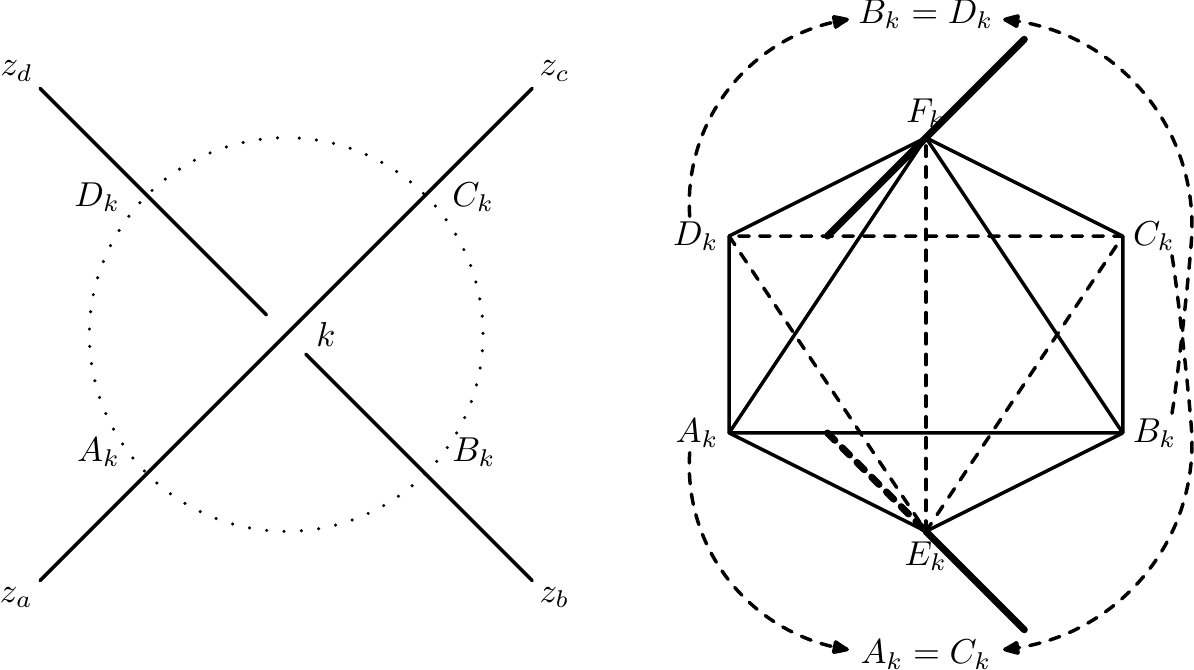}}
\caption{Octahedron on the crossing k.}
\label{fig:order}
\end{center} 
\end{figure}
\begin{figure} 
\begin{center}
\resizebox{5cm}{!}{\includegraphics{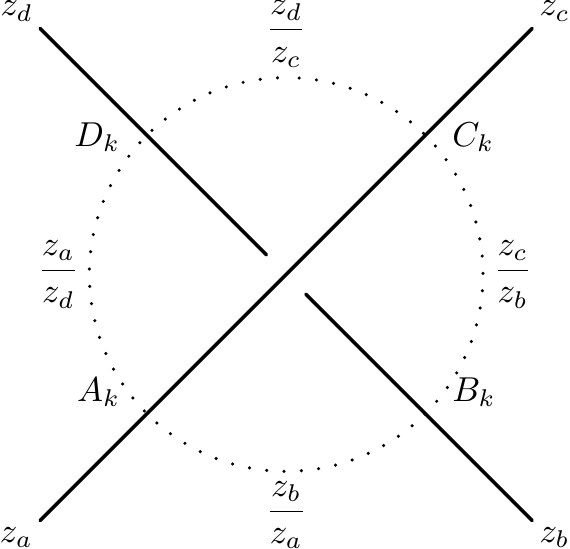}}
\caption{Shape parameters on the crossing k.}
\label{fig:shapeParameter}
\end{center} 
\end{figure}

\subsection{Potential function}
The general reference for this subsection is~\cite{CKK1}.
Let $D$ be a reduced planar projection of an alternating knot $K$ with $n/2$ crossings with $n$ even. For each crossing of $D$, we define the \emph{potential function of a crossing} as the sum of four terms as in Figure~\ref{fig:potential}.
\begin{figure} 
\begin{center}
\includegraphics{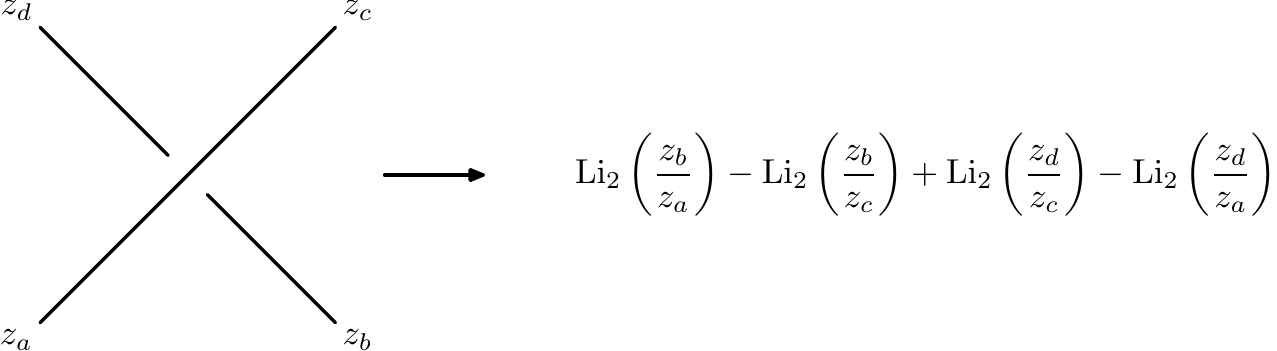}
\caption{The potential function for a crossing.}
\label{fig:potential}
\end{center} 
\end{figure}

\begin{definition} \label{def:potential}
The potential function $V(z_1, \ldots,z_{n})$ of the diagram $D$ is defined as the summation of all potential functions of the $n/2$ crossings.
\end{definition}

Define 
\begin{equation}
\mathcal{H}:= \left \{exp \left(z_k \frac{\partial V}{\partial z_k}\right)=exp \left(\pm \frac{2 \pi i}{r} \right) \left |\ k=1, \ldots, n \right. \right \}
\end{equation}
where $+$ corresponds to Type I of Figure~\ref{fig:side} and $-$ corresponds to Type II of Figure~\ref{fig:side}.

\begin{figure} 
\begin{center}
\resizebox{10cm}{!}{\includegraphics{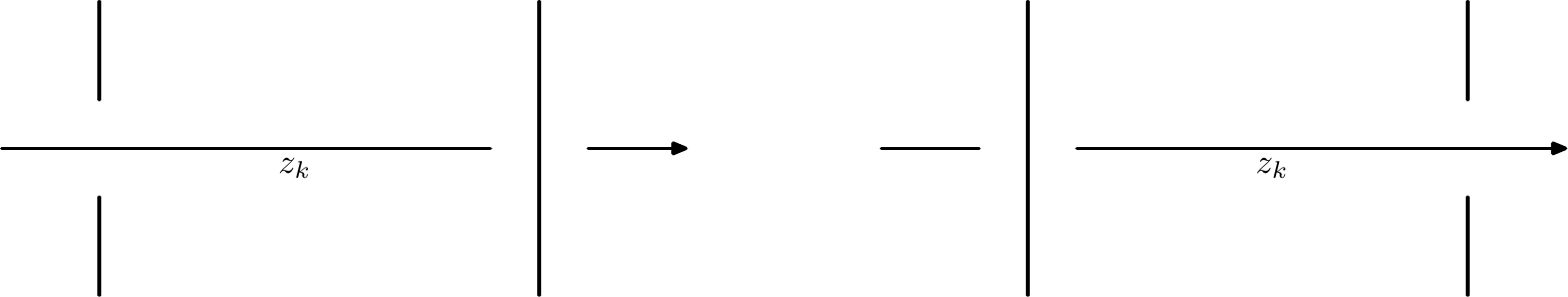}}
\caption{The side of Type I (left) and the side of Type II (right)}
\label{fig:side}
\end{center} 
\end{figure}

Then $\mathcal{H}$ gives the following set of equations which we also write $\mathcal{H}$ for notational convenience:

\begin{equation}
 \left \{exp \left(\sum_1^r z_k \frac{\partial V}{\partial z_k}\right)=1 \left |\ k=1, \ldots, n \right. \right \}.
\end{equation}

Recall that $M$ is the $r$-fold cyclic covering of $S^3 \backslash K$.
Then $\mathcal{H}$ becomes the hyperbolicity equations of $M$~\cite[Proposition 1.1]{CKK1}. In~\cite{CKK1}, They proved it for knots. But the similar argument works for cyclic coverings, too. Note that the set of equations $\mathcal{H}$ consists of the completeness conditions along the meridian
and this gives Thurston's gluing equations for $M$ using similar argument to~\cite[Lemma 3.1]{CKK1}. 

Using the formula in Theorem~\ref{thm:main}, we can compute
\begin{align*}
V (z_1, . . . , z_{n}) -\sum_{k=1}^{n} \left(z_k \frac{\partial V}{\partial z_k} \right) \log{z_k} \qquad \\
\left(\textnormal{mod } \frac{\pi^2}{r} \textnormal{ for } r \textnormal{ odd and } \textnormal{mod } \frac{2 \pi^2}{r} \textnormal{ for } r
 \textnormal{ even}\right).
\end{align*}
for any solution of $\mathcal{H}$. Similar argument to~\cite[Lemma 2.2]{CKK1} shows that the value of the formula of a solution of 
$\mathcal{H}$ remains constant for each component of the solution set of $\mathcal{H}$.
Among the solution set of $\mathcal{H}$, denote the component which gives the geometric structure by $\mathcal{S}$.

\section{complex volume of orbifolds}
Let $M'$ be the completion of $M$. Then $M'$ is a compact manifold. $M'$ can be obtained from the hyperbolic manifold $M$ by a hyperbolic Dehn filling~\cite{K88,Sak91}.
In~\cite[Theorem 14.5 and Theorem 14.7]{N2}, Neumann introduced two ways to compute the complex volume of a Dehn filled manifold. 
For our triangulation the following Theorem~\ref{thm:Neumann} can be easily applied when combined with Zickert's methods in~\cite{Z1,Z16,GZ17}.

\begin{theorem}~\cite[Theorem 14.7]{N2} \label{thm:Neumann}
Consider the degree one triangulation of $M′$. Then there exists flattenings $[x_i ; p_i, q_i]$ of the ideal simplices which satisfy the
conditions
\begin{enumerate}
\item[(i)] parity along normal paths is zero;
\item[(ii)] log-parameter about each edge is zero;
\item[(iii)] log-parameter along any normal path in the neighborhood of a $0$-simplex that represents an unfilled cusp is zero;
\item[(iv)] log-parameter along a normal path in the neighborhood of a $0$-simplex that represents a filled cusp is zero if the path is null-homotopic in the added solid torus.
\end{enumerate}
For any such choice of flattenings we have
\begin{align*}
\widehat{\beta} (M')=\sum_i \epsilon_i [x_i ; p_i , q_i] \in \widehat{\B} (\C)
\end{align*}
so
\begin{align} \label{equ:cxN}
i (\textnormal{vol}+i \, \textnormal{cs})(M')=\sum_i \epsilon_i R(x_i ; p_i, q_i).
\end{align}
\end{theorem}

Since we are using a $P$-decoration of the geometric representation on $\widetilde{\overline{\T}}$  (liftings of $\overline{\T}$  to the universal cover of $M$) to compute complex volume, where $P$ is the subgroup of upper triangular matrices with $1$'s on the diagonal, we can apply Zickert's methods in~\cite{Z1,Z16,GZ17} to our octahedral triangulation to find flattening parameters 
$[x_i^{′}; p_i^{′}, q_i^{′}]$ for $M'$ which satisfy the conditions of Theorem~\ref{thm:Neumann}  except the condition (i) and (iii)~\cite[Theorem 6.5]{Z1}. Since $M'$ does not have unfilled cusp, condition $(iii)$ of Theorem~\ref{thm:Neumann} will be automatically satisfied. Denote $M'=M_r(K)$.
Without the condition (i) of Theorem~\ref{thm:Neumann}, the equation~(\ref{equ:cxN}) of Theorem~\ref{thm:Neumann} will still give the complex volume at worst modulo $\pi^2/2$~\cite[Lemma 11.3]{N2}. But since $R$ in Theorem~\ref{thm:main} and the volume function and the Chern-Simons function  in ~\cite[Theorem 3.9]{HLM3} are analytic and since with our triangulation the equation~(\ref{equ:cxN}) of Theorem~\ref{thm:Neumann} gives the complex volume modulo $\pi^2$ at $\Oc(K,\infty)$ when regarded as a knot complement~\cite[Theorem 1.2]{CKK1}, $R$ will give the complex volume modulo $\pi^2/r$ at $\Oc(K,r)$ for $r$ odd and modulo $2 \pi^2/r$ at $\Oc(K,r)$ for $r$ even and hence modulo $\pi^2$ at $M_r(K)$ for $r$ odd and modulo $2 \pi^2$ at $M_r(K)$ for $r$ even.
With our octahedral triangulation, the equation~(\ref{equ:cxN}) of Theorem~\ref{thm:Neumann} gives the equation of the following Theorem~\ref{thm:main} which is our main theorem.

\begin{theorem} \label{thm:main}
Let $K$ be a hyperbolic alternating knot with a reduced diagram and $V(z_1,...,z_{n})$ be the potential function of the diagram. Then for any 
$\z \in \Sc$ which will give the maximal volume, we have
\begin{align*}
i (\textnormal{vol}+i \, \textnormal{cs})(\Oc(K,r))=V (z_1, . . . , z_{n}) -\sum_{k=1}^{n} \left(z_k \frac{\partial V}{\partial z_k} \right) \log{z_k} \qquad \\
\left(\textnormal{mod } \frac{\pi^2}{r} \textnormal{ for } r \textnormal{ odd and } \textnormal{mod } \frac{2 \pi^2}{r} \textnormal{ for } r
 \textnormal{ even}\right).
\end{align*}
\end{theorem}
\section{The proof of Theorem~\ref{thm:main}} \label{sec:proof}
The proof is parallel with that of Theorem 1.2 of~\cite{CKK1}. In this section, we always assume $z = (z_1, . . . , z_{n})$ is a solution in $\Sc$ which will give the geometric structure. $z$ for one sheet is different from another sheet but since 
the image of the fundamental class, $[\overline{M}]$, in $\widehat{\B}(\C)$ can be expressed by using the same $z$ for each sheet, we use the same $z$ for each sheet.  
To apply the methods in~\cite{N2, Z1,Z16,GZ17}, we use the same vertex ordering as in~\cite{CKK1}. To the vertices of Figure~\ref{fig:label}, we assign vertex orders from $0$ to $3$ to the vertices $\left(E_k,F_k,A_k(C_k),B_k(D_k)\right)$ in order. This assignment induces the vertex orderings of the four tetrahedra.
If the orientation of a tetrahedron from the vertex ordering is the same as the ambient orientation, we assign  the sign of the tetrahedron $\sigma = 1$, otherwise $\sigma = -1$. Each tetrahedron appears in the extended pre-Bloch group as $\sigma[u^{\sigma};p,q] \in \widehat{\Pc} (\C)$, where $\sigma$ is the sign of the tetrahedron, $u$ is the shape parameter assigned to the edge connecting the $0$th and $1$st vertices, and $p$, $q$ are certain integers.

Zickert used a \emph{Ptolemy coordinate} $c_{ij}$ for each edge to determine $p,\ q$ of 
$\sigma[u^{\sigma};p,q] \in \widehat{\Pc} (\C)$~\cite[Equation 3.6]{Z1}:
\begin{equation}
\begin{cases}
p \pi i=-\log{u^\sigma} +\log{c_{03}} +\log{c_{12}} -\log{c_{02}} -\log{c_{13}}, \\
q \pi i=\log{(1-u^{\sigma})}+\log{c_{02}} +\log{c_{13}} -\log{c_{23}} -\log{c_{01}}.
\end{cases}
\label{equ:pq}
\end{equation}

We use Ptolemy coordinates, but we don't need exact values since they all cancel out eventually.

\begin{figure} 
\begin{center}
\resizebox{10cm}{!}{\includegraphics{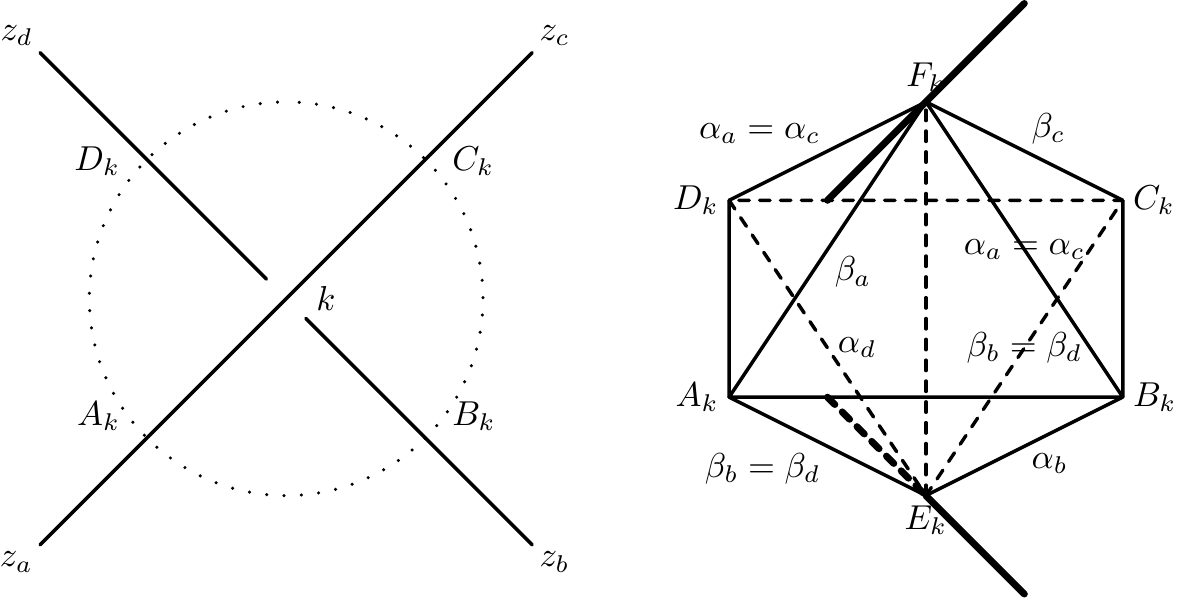}}
\caption{Labelings of non-horizontal edges.}
\label{fig:label}
\end{center} 
\end{figure}

\begin{figure} 
\begin{center}
\resizebox{10cm}{!}{\includegraphics{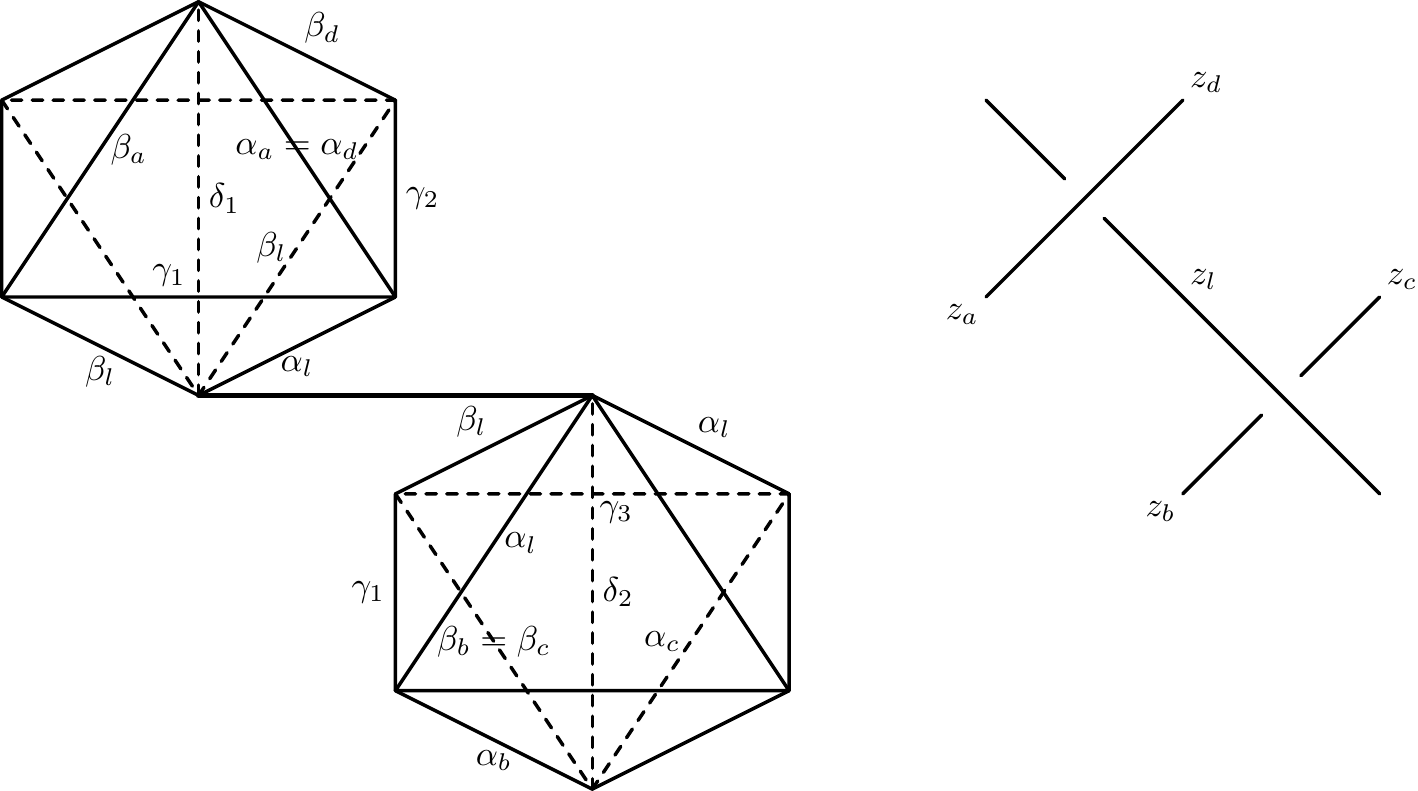}}
\caption{Gluing two octahedra.}
\label{fig:label1}
\end{center} 
\end{figure}

In~\cite{CKK1}, they used $\alpha_m,\ \beta_m,\ \gamma_l,\ \delta_k$ for $c_{ij}$. They assigned $\alpha_m$ and 
$\beta_m$ to non-horizontal edges, $\gamma_l$ to horizontal edges and $\delta_k$ to the edge $E_kF_k$ inside the octahedron as in Figure~\ref{fig:label} and Figure~\ref{fig:label1}, where $m = a, b, c, d$.  Although $\alpha_a = \alpha_c$ and $\beta_b = \beta_d$, they used $\alpha_a$ for the tetrahedron $E_kF_kA_kB_k$ and $E_kF_kA_kD_k$, $\alpha_c$ for $E_kF_kC_kB_k$ and $E_kF_kC_kD_k$, $\beta_b$ for $E_kF_kA_kB_k$ and $E_kF_kC_kB_k$, 
$\beta_d$ for 
$E_kF_kC_kD_k$ and $E_kF_kA_kD_k$. We follow their convention. Since we assigned vertex orders from $0$ to $3$ to the vertices 
$\left(E_k,F_k,A_k(C_k),B_k(D_k)\right)$ of Figure~\ref{fig:label} in order, the orientation of the octahedral triangulation induced by this ordering satisfies the edge-orientation consistency; when two edges are glued together in the triangulation, the orientations of the two edges induced by each vertex orderings coincide.

Even though Observation 4.1 of~\cite{CKK1} is for an octahedron of the octahedral triangulation of a link complement, the proof only uses the fact that they are using a $P$-decoration of the geometric representation. Hence we have the following Proposition~\ref{prop:const} which is an extension of Observation 4.1 of~\cite{CKK1} to our $M$.
\begin{proposition} [Observation 4.1 of~\cite{CKK1}] \label{prop:const}
For a hyperbolic reduced alternating knot diagram of a sheet of $M$ with the octahedral triangulation, we have
$$\log{\alpha_l} -\log{\beta_l} \equiv \log{z_l }+A \qquad (\textnormal{mod } \pi i),$$
for all $l = 1,\ldots,n$, where $n$ is the number of sides of the diagram and $A$ is a complex
constant number independent of $l$.
\end{proposition}

For $m = a,\ b,\ c,\ d$ and for the tetrahedron between the sides $z_l$ and $z_m$, for the sign $\sigma$, we use $\sigma_l^m$. Similarly, we use $u_l^m$ (the shape parameter), $p_l^m$ and $q_l^m$. We define $\tau_l^m = 1$ when $z_l$ is the numerator of $(u_l^m )^{\sigma_l^m}$ and $\tau_l^m = -1$ otherwise. Now, each tetrahedron appears in $\widehat{\Pc}(\C)$ as 
$\sigma_l^m[(u_l)^{\sigma_l^m} ; p_l^m , q_l^m]$.
By definition, we know
\begin{equation}
u_l^a=\frac{z_l}{z_a}, \quad u_l^b=\frac{z_b}{z_l}, \quad u_l^c=\frac{z_l}{z_c}, \quad u_l^d=\frac{z_d}{z_l}.
\label{equ:sol}
\end{equation}
and
$$\sigma_l^a =1,\ \sigma_l^b =1,\ \sigma_l^c =-1,\ \sigma_l^d =-1\ \ \text{and} \ \ \tau_l^a =1,\ \tau_l^b =-1,\ \tau_l^c =-1,\ \tau_l^d =1.$$
Using Equation~\ref{equ:pq} and Figure~\ref{fig:label1}, we decide $p_l^m$ and $q_l^m$ as follows:
\begin{equation}
\begin{cases}
\log{\frac{z_l}{z_a}} +p_l^a \pi i= \log{\alpha_l} +\log{\beta_a} -\log{\beta_l} -\log{\alpha_a}, \\
\log{\frac{z_b}{z_l}} +p_l^b \pi i= \log{\alpha_b} +\log{\beta_l} -\log{\beta_b} -\log{\alpha_l}, \\
\log{\frac{z_c}{z_l}} +p_l^c \pi i= \log{\alpha_c} +\log{\beta_l} -\log{\beta_c} -\log{\alpha_l}, \\
\log{\frac{z_l}{z_d}} +p_l^d \pi i= \log{\alpha_l} +\log{\beta_d} -\log{\beta_l} -\log{\alpha_d},
\end{cases}
\label{equ:ps}
\end{equation}
\begin{equation}
\begin{cases}
-\log{(1-\frac{z_l}{z_a})} +q_l^a \pi i= \log{\beta_l} +\log{\alpha_a} -\log{\gamma_1} -\log{\delta_1}, \\
-\log{(1-\frac{z_b}{z_l})} +q_l^b \pi i=  \log{\beta_b} +\log{\alpha_l} -\log{\gamma_1} -\log{\delta_2}, \\
-\log{(1-\frac{z_c}{z_l})} +q_l^c \pi i=  \log{\beta_c} +\log{\alpha_l} -\log{\gamma_3} -\log{\delta_2}, \\
-\log{(1-\frac{z_l}{z_d})} +q_l^d \pi i=  \log{\beta_l} +\log{\alpha_d} -\log{\gamma_2} -\log{\delta_1},
\end{cases}
\label{equ:qs}
\end{equation}
Note that
$$\sigma_l^m =\sigma_m^l,\  \tau_l^m =-\tau_m^l,\  u_l^m =u_m^l,\  p_l^m =p_m^l,\  q_l^m =q_m^l\ \ \text{and}$$
$$\sigma_l^m[(u_l^m )^{\sigma_l ^m}; p_l^m, q_l^m] = \sigma_m^l [(u_m^l)^{\sigma_m^l} ; p_m^l, q_m^l ] \in \Pc(\C).$$

The image of the fundamental class, $[M^{\prime}]$, in $\widehat{\B}(\C)$ 
can be written as 
$$\frac{1}{2} \sum_1^r\sum_{l,m} \sigma_l^m[(u_l^m )^{\sigma_l ^m}; p_l^m , q_l^m]$$ using our triangulation.
Hence, the potential function $V$ can be written
$$V(z_1, \ldots, z_n)=\frac{1}{2} \sum_{l,m} \sigma_l^m \text{Li}_2((u_l^m )^{\sigma_l^m}).$$
By direct calculation, we obtain
\begin{equation}
z_l \frac{\partial V}{\partial z_l}=-\sum_{m=a, \ldots, d} \sigma_l^m \tau_l^m \log{(1-(u_l^m )^{\sigma_l^m})}
\label{equ:meridian}
\end{equation}
for all $l=1, \ldots, n$.

\begin{lemma} \label{lem:meridian}
For all $l=1, \ldots, n$, we have
$$z_l \frac{\partial V}{\partial z_l}=-\sum_{m=a, \ldots, d} \sigma_l^m \tau_l^m q_l^m \pi i+\log{\gamma_2}-\log{\gamma_{3}}.$$
and hence
$$\sum_1^r z_l \frac{\partial V}{\partial z_l}=-\sum_1^r \sum_{m=a, \ldots, d} \sigma_l^m \tau_l^m q_l^m \pi i.$$
\end{lemma}
\begin{proof}
Using Equation~(\ref{equ:qs}), Equation~(\ref{equ:meridian}), $\alpha_a=\alpha_d$ and $\beta_b=\beta_c$,
we can directly calculate the following:
\begin{align*}
z_l \frac{\partial V}{\partial z_l}&=-\sum_{m=a, \ldots, d} \sigma_l^m \tau_l^m \log{(1-(u_l^m )^{\sigma_l^m})}\\
&=-q_l^a \pi i+q_l^b \pi i-q_l^c \pi i+q_l^d \pi i+\log{\gamma_2}-\log{\gamma_3}.
\label{equ:meridian1}
\end{align*}
Note that $\gamma_3$ of $i$-th sheet and $\gamma_2$ of $(i+1)$-th sheet are the same because they are glued together in the triangulation. Hence we get the second equality.
\end{proof}

\begin{lemma} \label{lem:mainP}
For all possible $l$ and $m$, we have
\begin{equation*}
\frac{1}{2} \sum_1^r \sum_{l,m} \sigma_l^m q_l^m \pi i \log{(u_l^m )^{\sigma_l^m}} 
\equiv -\sum_1^r \sum_{l=1}^n \left(z_l \frac{\partial V}{\partial z_l}\right) \log{z_l} \quad \left(\textnormal{mod } 2 \pi^2\right).
\end{equation*}
\end{lemma}
\begin{proof}
Note that $q_l^m$ is an integer. Using Equation~(\ref{equ:sol}) and Lemma~\ref{lem:meridian}, we can directly calculate
\begin{align*}
\frac{1}{2} \sum_1^r \sum_{l=1}^n \sum_{m=a, \ldots, d} \sigma_l^m q_l^m \pi i \log{(u_l^m )^{\sigma_l^m}} 
& \equiv \sum_1^r \sum_{l=1}^n\left(\sum_{m=a, \ldots, d} \sigma_l^m \tau_l^m q_l^m \pi i\right) \log{z_l} 
\quad \left(\textnormal{mod } 2 \pi^2\right)\\
&=-\sum_1^r \sum_{l=1}^n\left(z_l \frac{\partial V}{\partial z_l}-\log{\gamma_2}+\log{\gamma_3}\right) \log{z_l} \\
&=-\sum_1^r \sum_{l=1}^n \left(z_l \frac{\partial V}{\partial z_l}\right) \log{z_l}.
\end{align*}
\end{proof}

\begin{lemma} \label{lem:main1}
For all possible l and m, we have
\begin{equation*}
\frac{1}{2} \sum_1^r \sum_{l,m} \sigma_l^m \log{(1-(u_l^m )^{\sigma_l^m})}(\log{(u_l^m )^{\sigma_l^m}}+p_l^m \pi i)
\equiv - \sum_1^r \sum_{l=1}^n \left(z_l \frac{\partial V}{\partial z_l}\right) \log{z_l} \quad \left(\textnormal{mod } 2 \pi^2\right). 
\end{equation*}
\end{lemma}
\begin{proof}
By applying
Equation~(\ref{equ:ps}) and Equation~(\ref{equ:meridian}), we can verify
\begin{align*}
&\frac{1}{2}  \sum_1^r \sum_{l,m} \sigma_l^m \log{(1-(u_l^m )^{\sigma_l^m})}(\log{(u_l^m )^{\sigma_l^m}}+p_l^m \pi i)\\
&=  \sum_1^r \sum_{l=1}^n \left(\sum_{m=a, \ldots, d} \sigma_l^m \tau_l^m \log{(1-(u_l^m )^{\sigma_l^m})}\right) 
(\log{\alpha_l}-\log{\beta_l})\\
&=-  \sum_1^r \sum_{l=1}^n \left(z_l \frac{\partial V}{\partial z_l}\right) (\log{\alpha_l}-\log{\beta_l})\\
&\equiv -  \sum_1^r \sum_{l=1}^n \left(z_l \frac{\partial V}{\partial z_l}\right) (\log{z_l}+A) \quad  \left(\textnormal{mod } 2 \pi^2\right) 
\quad \textnormal{(by Proposition~\ref{prop:const})}\\
&= -  \sum_1^r \sum_{l=1}^n \left(z_l \frac{\partial V}{\partial z_l}\right) \log{z_l} 
\quad \left(\textnormal{by}  \sum_1^r \sum_{l=1}^n \left(z_l \frac{\partial V}{\partial z_l}\right)=0\right).\\
\end{align*}
\end{proof}

Combining Equation~(\ref{equ:cxN}), Lemma~\ref{lem:mainP}, and Lemma~\ref{lem:main1}, we prove our main theorem, Theorem~\ref{thm:main}:
\begin{align*}
i (\textnormal{vol}+i \, \textnormal{cs})(M)=i (\textnormal{vol}+i \, \textnormal{cs})(M') = R \left(\frac{1}{2} \sum_1^r \sum_{l,m} \sigma_l^m[(u_l^m )^{\sigma_l^m};p_l^m,q_l^m]\right)\\
=\frac{1}{2} \sum_1^r \sum_{l,m} \sigma_l^m \left(\text{Li}_2\left((u_l^m )^{\sigma_l^m}\right)
-\frac{\pi^2}{6}\right)
+\frac{1}{4} \sum_1^r \sum_{l,m} \sigma_l^m q_l^m \pi i \log{(u_l^m )^{\sigma_l^m}} \\
+\frac{1}{4} \sum_1^r \sum_{l,m} \sigma_l^m \log{(1-(u_l^m )^{\sigma_l^m})}(\log{(u_l^m )^{\sigma_l^m}}+p_l^m \pi i)\\
= r V (z_1, . . . , z_n)-r\sum_{k=1}^n \left(z_k \frac{\partial V}{\partial z_k} \right) \log{z_k} \qquad  
\left(\textnormal{mod } \pi^2\right).
\end{align*}
Hence we have
\begin{align*}
i (\textnormal{vol}+i \, \textnormal{cs})(\Oc(K,r))=V (z_1, . . . , z_n) -\sum_{k=1}^n \left(z_k \frac{\partial V}{\partial z_k} \right) \log{z_k} \qquad \\
\left(\textnormal{mod } \frac{\pi^2}{r} \textnormal{ for } r \textnormal{ odd and } \textnormal{mod } \frac{2 \pi^2}{r} \textnormal{ for } r
 \textnormal{ even}\right).
\end{align*}

\section{The shape of $J(2n,-2m)$ knot orbifolds and a root of the Riley-Mednykh polynomial} \label{sec:shape}
 In this section, we present the Riley-Mednykh polynomial of $J(2n,-2m)$ knot. We show that a root of a certain polynomial, which shares a root with the Riley-Mednykh polynomial (see Subsection~\ref{subsec:RMY}), determines the shape of 
 $J(2n,-2m)$ knot orbifolds. Theorem~\ref{thm:solution} states that a root of a certain polynomial gives a solution set for the system of hyperbolicity equations of $J(2n,-2m)$, and this, in turn, gives a set of shape parameters corresponding to ideal tetrahedra of a triangulated $J(2n,-2m)$ knot orbifold and the potential function of it (Definition~\ref{def:potential}) so that we can get the complex volume of it using Theorem~\ref{thm:main} in Section~\ref{sec:cvolume}.
\subsection{J(2n,-2m) knot} \label{subsec:Jnm}
A general reference for this section is ~\cite{HS}. 
A knot is $J(2n,-2m)$ knot if it has a regular two-dimensional projection of the form in Figure~\ref{fig:diagrams}. Note that $n$ and $m$ in this subsection is the half of $n$ and $m$ of Subsection~\ref{subsec:solution}, respectively.
It has 2n right-handed vertical crossings ($n$ right-handed vertical full twists) and $2 n$ left-handed horizontal crossings ($n$ left-handed horizontal full twists). 

We will use the following fundamental group of $J(2n,-2m)$ in~\cite{HS}. Let$X_{2n}^{2m}$ be $S^3 \backslash J(2n,-2m)$.

\begin{proposition}\label{theorem:fundamentalGroup}
$$\pi_1(X_{2n}^{2m})=\left\langle s,t \ |\ sw^mt^{-1}w^{-m}=1\right\rangle,$$
where $w=(t^{-1}s)^n(ts^{-1})^n$.
\end{proposition}
\subsection{The Chebychev polynomial}\label{subsec:cheby}
Let $S_k(\xi)$ be the \emph{Chebychev polynomials} defined by $S_0(\xi) = 1$, $S_1(\xi) = \xi$ and $S_k(\xi) = \xi S_{k-1}(\xi) - S_{k-2}(\xi)$ for all integers $k$.
The following explicit formula for $S_k(v)$ can be obtained by solving the above recurrence relation~\cite{Tran2}.
\begin{align*}
S_n(\xi)=\sum_{0 \leq i \leq \lfloor \frac{n}{2} \rfloor} (-1)^i \binom{n-i}{i} \xi^{n-2i}
\end{align*}
for $n \geq 0$, $S_n(\xi) = -S_{-n-2}(\xi)$ for $n \leq -2$, and $S_{-1}(\xi) = 0$.
The following proposition~\ref{prop:cheby} can be proved using the Cayley-Hamilton theorem~\cite{Tran}.
\begin{proposition}~\cite[Proposition 2.4]{Tran}\label{prop:cheby}
Suppose $V=
\begin{bmatrix}
a & b \\ c & d
\end{bmatrix}
\in \textnormal{SL}_2(\C).$
Then
\begin{equation*}
V^k=
\begin{bmatrix}
S_k(\xi)-d S_{k-1}(\xi) & b S_{k-1} (\xi) \\ c S_{k-1} (\xi) & S_k(\xi)-a S_{k-1} (\xi)
\end{bmatrix}
\end{equation*}
where $\xi=\textnormal{tr}(V)=a+d$.
\end{proposition}

\subsection{The Riley-Mednykh polynomial}
Let
\begin{center}
$$\begin{array}{ccccc}
\rho(s)=\left[\begin{array}{cc}
                       M &       1 \\
                        0      & M^{-1}  
                     \end{array} \right]                          
\text{,} \ \ \
\rho(t)=\left[\begin{array}{cc}
                   M &  0      \\
                   2-M^2-M^{-2}-x      & M^{-1} 
                 \end{array}  \right],
\end{array}$$
\end{center}
and let
 \begin{center}
$$\begin{array}{cc}
c=\left[\begin{array}{cc}
        0 & -\left(\sqrt{2-M^2-M^{-2}-x }\right)^{-1}     \\
       \sqrt{2-M^2-M^{-2}-x} & 0
       \end{array}  \right].
\end{array}$$
\end{center}

Then from the above Proposition~\ref{prop:cheby}, we get the following Theorem~\ref{thm:RMpolynomial}. 
Let $\rho(s)=S$, $\rho(t)=T$ and 
$\rho(w)=W$. 
Then $\textnormal{tr}(T^{-1}S)=x+M^2+M^{-2}=\textnormal{tr}(TS^{-1})$. Let $v=x+M^2+M^{-2}$ and Let $z=\text{Tr}(W)$.

\begin{theorem} \label{thm:RMpolynomial}
$\rho$ is a representation of $\pi_1(X_{2n}^{2m})$ if and only if $x$ is a root of the following Riley-Mednykh polynomial,
\begin{align*}
\phi_{2n}^{2m} (x,M)=S_m(z)+\left[-1+x  S_{n-1}(v)\left(S_n(v)+(1-v) S_{n-1}(v)\right)\right]S_{m-1}(z).
\end{align*}
\end{theorem}

\begin{proof}
Since
\begin{align*}
T^{-1}S &=\left[\begin{array}{cc}
                       1 &      M^{-1}  \\
                       M( M^{-2}+(x-2) +M^2)\ \ \ \ \    & M^{-2} +x-1+M^2 
                     \end{array} \right]  \text{ and} \\
T S^{-1} &=\left[\begin{array}{cc}
                   1  &  -M      \\
                   -M^{-1} (M^{-2}+(x-2) +M^2)    & M^{-2} +x-1+M^2 
                 \end{array}  \right],
\end{align*}

\begin{align*}
(T^{-1} S)^n &=\left[\begin{array}{cc}
                       S_n(v)-(v-1) S_{n-1}(v) &      M^{-1} S_{n-1}(v) \\
                        M (v-2) S_{n-1}(v)  \ \ \ \ \  & S_n(v)-S_{n-1}(v)  
                     \end{array} \right] , \\    
(T S^{-1})^n &=\left[\begin{array}{cc}
                       S_n(v)-(v-1) S_{n-1}(v) &      -M S_{n-1}(v) \\
                        -M^{-1} (v-2) S_{n-1}(v)     & S_n(v)-S_{n-1}(v)  
                     \end{array} \right]     
\end{align*}

Hence 
\begin{equation*}
W=(T^{-1} S)^n (T S^{-1})^n 
=\begin{bmatrix}
W_{11} & W_{12} \\ (2-v) W_{12} & W_{22}
\end{bmatrix}
\ \text{where}
\end{equation*}

\begin{align*}
W_{11} & =S^2_n(v)+(2-2v) S_n(v) S_{n-1}(v)+(1+2 M^{-2}-2v-M^{-2}v+v^2) S^2_{n-1}(v) \\
W_{12} & =(M^{-1}-M) S_n(v) S_{n-1}(v)+(M v-M-M^{-1}) S^2_{n-1}(v)\\
W_{22} &=S^2_n(v)-2 S_n(v) S_{n-1}(v)+(1+2 M^2-M^2 v) S^2_{n-1}(v).
\end{align*}

Then, since $S_n^2 (v) -vS_n(v)S_{n-1}(v) + S_{n-1}^2(v) = 1$ (by~\cite[Lemma 2.1]{Tran1} or by induction),

\begin{align*}
z=W_{11}+W_{22} &=2(S^2_n(v)-v S_n(v) S_{n-1}(v)+S^2_{n-1}(v))\\
&+(2 M^{-2}+2 M^2-2v-M^{-2}v-M^2 v+v^2)S^2_{n-1}(v)\\
&=2+(v-2) x S^2_{n-1}(v).
\end{align*}

By Proposition~\ref{prop:cheby}, we have
\begin{align*}
(W)^m &=\left[\begin{array}{cc}
                       S_m(z)-W_{22} S_{m-1}(z) &     W_{12}  S_{m-1}(z) \\
                         (2-v) W_{12} S_{m-1}(z)  \ \ \ \ \  & S_m(z)-W_{11} S_{m-1}(z)  
                     \end{array} \right]. 
\end{align*}                       

Therefore $\textnormal{tr}(SW^mc)/\sqrt{2-v}\left(=(M-M^{-1}) W_{12}  S_{m-1}(z) +S_m(z)-W_{11} S_{m-1}(z)\right)$ gives 
$\phi_{2n}^{2m} (x,M)$~\cite{HLMR}.
\end{proof}
\subsection{$M^2$-deformed solutions for $J(2n,-2m)$} \label{subsec:solution}
Denote the solutions of $\mathcal{H}$ by $M^2$-deformed solutions. In this subsection we assume $J(2n,-2m)=J(n',-m')$ and then drop the prime ($'$) for notational convenience.
Let us assign segment variables $z_1,\cdots,z_{2n+2}, z'_1,\cdots,z'_{2m+2}$ as in Figure \ref{fig:diagrams}. 
Note that four segments are doubly labeled : $z'_1=z_{2n+1}, z'_2=z_1, z'_{2m+1}=z_{2n+2}$, and $z'_{2m+2}=z_2$.
\begin{figure}[H]
\centering
\resizebox{6cm}{!}{\includegraphics{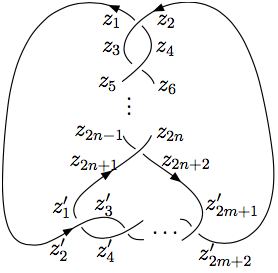}}
\caption{A diagram of $J(n,-m)$ with segment variables.}
\label{fig:diagrams}
\end{figure} 

For a rational polynomial $P$ of $M$ we define $\widetilde{P}$ by a rational polynomial obtained from $P$ by replacing $M$ with 
$\dfrac{1}{M}$. Let $(B_j)_{j \in \Zbb}$ be the sequence defined by the following recurrence relations 
\begin{equation} \label{eqn:B}
\left \{
\begin{array}{rclr}
B_{j+1}&=& \sqrt{\Lambda} B_j +M^{2} B_{j-1}  &\textrm{ for odd } j \\
B_{j+1} &=& \sqrt{\Lambda} B_j +M^{-2} B_{j-1} & \textrm{ for even } j 			
\end{array} 
\right. 
\end{equation} 
with initial conditions $B_0=0,\ B_1=1$. Let sequences $(P_j)_{j \in \Zbb}$ and $(Q_j)_{j \in \Zbb}$ be also defined by the same recurrence relations as $B_j$ with initial conditions $P_0=z_1(z_2- \dfrac{z_3}{M^2}),\ P_1 =\sqrt{\Lambda}z_2 z_3$\ and\ $Q_1=\sqrt{\Lambda} z_3,\ Q_2=z_2-\dfrac{z_3}{M^2}$.  
Let $(B'_j)_{j \in \Zbb}$ be the sequence defined by the following recurrence relations:
\begin{equation}\label{eqn:BB}
\left \{
\begin{array}{rclr}
 B'_{j+1} &=&  \widetilde{(B_{n+1}}-\widetilde{B_{n-1})}B'_j +  M^{-2} B'_{j-1} & \textrm{ for odd } j \\[5pt]
 B'_{j+1} &=& (B_{n+1}-B_{n-1}) B'_j + M^{2} B'_{j-1}  &\textrm{ for even } j		
\end{array} 
\right. 
\end{equation} 
with initial condition $B'_0=0,\ B'_1=1$. 
Let sequences $(P'_j)_{j \in \Zbb}$ and $(Q'_j)_{j \in \Zbb}$ be defined by the same recurrence relations as $B'_j$ with initial conditions $P'_0=P_n,\  P'_1 =P_0$\ and\ $Q'_1=Q_2,\ Q'_2=Q_{n+2}$.

\begin{theorem} \label{thm:solution}
An $M^2$-deformed solution $(z_1,\cdots, z_{2n+2},\ z'_1,\cdots,\ z'_{2m+2})$ for the $J(n,-m)$ knot is given by
$$z_{2j-1}=\dfrac{P_{j-1}}{Q_{j+1}},\  z_{2j} =\dfrac{P_j}{Q_j},\ z'_{2j'-1}=\dfrac{P'_{j'-1}}{Q'_{j'+1}},\ \textrm{ and }\  z'_{2j'} =\dfrac{P'_{j'}}{Q'_{j'}}$$
for $1 \leq j \leq n+1, \ 1 \leq j' \leq m+1$ where $\Lambda$ satisfies  $B'_{m+1}+\widetilde{B'_m} \widetilde{B_{n-1}}=0$.
\end{theorem} 
\subsection{Proof of Theorem \ref{thm:solution}}
\begin{lemma} \label{lem:main} 
Let $(F_j)_{j \in \Zbb}$ be a sequence satisfying the following recurrence relations 	
\begin{equation*}
\left \{
\begin{array}{rclr}
F_{j+1} &=& W_1 F_j + M^{2} F_{j-1} & \textrm{ for odd } j \\
F_{j+1} &=& W_2 F_j + M^{-2} F_{j-1} & \textrm{ for even } j 			
\end{array}
\right.
\end{equation*} for some complex numbers $W_1$ and $W_2$. Let $(G_j)_{j\in \Zbb}$ be a sequence satisifying the same recurrence relations as $F_j$. Let us consider a diagram with segment variables as in Figure \ref{fig:local_twisting}. Suppose 
	\begin{equation}\label{eqn:main_lem}
		z_{2j-1}=\dfrac{F_{j-1}}{G_{j+1}} \textrm{ and }\ z_{2j}=\dfrac{F_j}{G_j}
	\end{equation}
	hold for $j=1,2$. Then $M^2$-hyperbolicity equations for $z_3,z_4,\cdots,z_{2n-1},z_{2n}$ hold if and only if 
	Equations~(\ref{eqn:main_lem}) hold for all $1 \leq j \leq n+1$.
	\begin{figure}[H]
		\centering
		\resizebox{7cm}{!}{\includegraphics{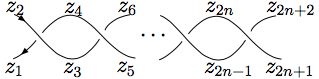}}
		\caption{Local diagrams with twisting.}
		\label{fig:local_twisting}
	\end{figure}
\end{lemma}

\begin{proof} For even $j$ from the $M^2$-hyperbolicity equations for $z_{2j-1}$ and $z_{2j}$, 
$$\dfrac{z_{2j-1}-z_{2j}}{z_{2j-1}-z_{2j-3}} \cdot \dfrac{z_{2j}(z_{2j-1}-z_{2j+1})}{z_{2j+1}(z_{2j-1}-z_{2j})}=M^2 \textrm{ and } \dfrac{z_{2j-2}(z_{2j}-z_{2j-1})}{z_{2j-1}(z_{2j}-z_{2j	-2})} \cdot \dfrac{z_{2j}-z_{2j+2}}{z_{2j}-z_{2j-1}}=M^{-2},$$ we obtain
	\begin{equation*}
		\left \{
			\begin{array}{ccc}
				z_{2j+1}&=&\dfrac{z_{2j-1} z_{2j}}{M^2z_{2j-1}+z_{2j} - M^2 z_{2j-3}}\\[10pt]
				z_{2j+2}&=& \dfrac{z_{2j-1}}{M^2}+ z_{2j} - \dfrac{z_{2j-1} z_{2j}}{M^2z_{2j-2}}								
			\end{array}
		\right.
	\end{equation*} 
Then one can directly check that  Equations~(\ref{eqn:main_lem}) hold for $j$ if they holds for $j-1$ and $j-2$. In a similar manner, one can check the case of odd $j$.
\end{proof}

Recall that we defined the sequences $(P_j)_{j \in \Zbb}, (Q_j)_{j \in \Zbb}$, and $(B_j)_{j \in \Zbb}$ by recursion relation (\ref{eqn:B}) with the initial conditions 
$Q_1= \sqrt{\Lambda} z_3, Q_2 = z_2-\dfrac{z_3}{M^2},\ P_0 =z_1(z_2-\dfrac{z_3}{M^2}), P_1 = \sqrt{\Lambda}z_2z_3$, and $B_0=0,B_1=1$. One can check that	
 Equations~(\ref{eqn:main_lem}) hold trivially for $j=1,2$. Therefore, by Lemma~\ref{lem:main}, Equations~(\ref{eqn:main_lem}) hold for $j=1,2,\cdots, n+1$. Note that $P_{j} = \dfrac{P_2}{\sqrt{\Lambda}} B_{j}- \dfrac{P_0}{\sqrt{\Lambda}} B_{j-2}$ holds for all $j$ and hence
\begin{equation} \label{eqn:explicit_PQ}
\left \{
\begin{array}{rcl}
Q_j&=&\dfrac{M^2 z_2-z_3}{M^2\sqrt{\Lambda}} B_j- \dfrac{M^2z_2-(1+M^2\Lambda)z_3}{M^4 \sqrt{\Lambda}} B_{j-2}\\[10pt]
P_j&=&\dfrac{\Lambda z_2 z_3 + z_1(M^2z_2-z_3)}{\sqrt{\Lambda}} B_j- \dfrac{z_1(M^2 z_2-z_3)}{M^2\sqrt{\Lambda}} B_{j-2}
\end{array}
\right.
\end{equation}

We take $z'_1=z_{2n+1}$ and $z'_2=z_1$, and then solve $M$-hyperbolicity equations for $z_{1}$ and $z_{2n+1}$,  
$\dfrac{z_3(z'_2-z_2)}{z_2(z'_2-z_3)}\cdot \dfrac{z'_2-z'_1}{z'_2-z'_4}=M^2$ and $ \dfrac{z'_2(z'_1-z'_3)}{z'_3(z'_1-z'_2)} \cdot \dfrac{z'_1-z_{2n-1}}{z'_1-z_{2n+2}} =M^2$, to obtain $z'_4$ and $z'_3$, respectively. One obtain
\begin{equation} \label{eqn:zz4}
z'_4=\dfrac{\widetilde{(B_{n+1}}-\widetilde{B_{n-1})} P_0 + M^{-2}P_n} {Q_{n+2}}
\end{equation} and
\begin{equation}\label{eqn:zz3}
z'_3=\dfrac{P_0}{(B_{n+1}-B_{n-1})Q_{n+2}+M^2Q_2}. 
\end{equation}
	
Now we take $P'_0=P_n,\ P'_1 =P_0$ and $Q'_1=Q_2,\ Q'_2=Q_{n+2}$, and define sequences $(P'_j)_{j \in \Zbb}$ and 
$(Q'_j)_{j \in \Zbb}$ by the recurrence relation (\ref{eqn:BB}). 
One can check that 
\begin{equation}\label{eqn:z'}
z'_{2j-1} = \dfrac{P'_{j-1}}{Q'_{j+1}} \textrm{ and } z'_{2j} =\dfrac{P'_j}{Q'_j}
\end{equation} 
hold for $j=1,2$. Therefore, by Lemma \ref{lem:main}, Equations~(\ref{eqn:z'}) hold for $j=1,\cdots,m+1$. 
	
We finally solve equations $z'_{2m+2}=z_2,\ z'_{2m+1}=z_{2n+2}$ and $M^2$-hyperbolicity equations for these segments.
One can check that $z'_{2m+2}=z_2$ if and only if 
\begin{equation} \label{eqn:last1}
(z_1-z_2)(z_2-\dfrac{z_3}{M^2})(B'_{m+1}+\widetilde{B'_m} \widetilde{B_{n-1}})=0.
\end{equation} Furthermore, one can check that $B'_{m+1}+\widetilde{B'_m} \widetilde{B_{n-1}}=0$ is necessary and sufficient condition for the remaining equations.
	
\subsection{Detailed computation}		
Equation (\ref{eqn:zz4}) : From $M^2$-hyperbolicity equation for $z_1$ we have $$z'_4= z_1 -\dfrac{z_3(z_1-z_2)(z_1-z_{2n+1})}{M^2 z_2(z_1-z_3)}.$$ On the other hand, we have $$z_1-z_{2n+1} = \dfrac{P_0}{Q_2}- \dfrac{P_n}{Q_{n+2}}=\dfrac{P_0 Q_{n+2}-P_n Q_2}{Q_2 Q_{n+2}} = \dfrac{(P_0 Q_3-P_1 Q_2)B_n}{Q_2 Q_{n+2}}= \dfrac{z_2(z_1-z_3)\sqrt{\Lambda} B_n }{Q_{n+2}}$$ and thus $z'_4= z_1-\dfrac{z_3(z_1-z_2) \sqrt{\Lambda} B_n}{M^2Q_{n+2}}$. Putting the explicit forms of Equations~(\ref{eqn:explicit_PQ}) of $P_j$ and $Q_j$ one results in Equation (\ref{eqn:zz4}).\\[10pt] 
	
Equation (\ref{eqn:zz3}) : From $M^2$-hyperbolicity equation  for $z_{2n+1}$ we have 
$$\dfrac{ Q_{n+2} P_0- \dfrac{P_n P_0}{z'_3}}{Q_{n+2}P_0 - Q_2 P_n}\cdot \dfrac{P_n Q_{n+1}-P_{n-1}Q_{n+2}}{P_n Q_{n+1} - P_{n+1}Q_{n+2}} =M^2.$$
One can check that $Q_{n+2}P_0 - Q_2P_n = (Q_3 P_0 - Q_2 P_1) B_n$ and $P_n Q_{n+1}-P_{n-1}Q_{n+2} = M^2(Q_3 P_0 -Q_2 P_1  )$ through induction. 
Thus the equation for $z_{2n+1}$ is further simplied to 
$$ Q_{n+2} P_0 - \dfrac{P_n P_0}{z'_3}= (P_n Q_{n+1}-P_{n+1} Q_{n+2})B_n.$$ 
We now put $z'_3=\dfrac{P_0}{(B_{n+1}-B_{n-1})Q_{n+2}+M^2 Q_2}$ and check  the equality holds : 
\begin{equation*}
\begin{array}{crcl}
&  Q_{n+2} P_0 - \dfrac{P_n P_0}{z'_3} &=& (P_n Q_{n+1}-P_{n+1} Q_{n+2})B_n \\
\Leftrightarrow&  Q_{n+2}P_0 -M^2 Q_2 P_n &=& P_n(Q_{n+1}B_n - Q_{n+2} B_{n-1})-Q_{n+2} (P_{n+1}B_n - P_n B_{n+1}) \\[3pt]
\Leftrightarrow& Q_{n+2}P_0 -M^2 Q_2 P_n  &=& P_n(Q_{1}B_0 - Q_{2} B_{-1})-Q_{n+2} (P_{1}B_{0} - P_{0}B_{1})
\end{array}
\end{equation*}\\[5pt]
	
Equation (\ref{eqn:last1}) : One can check that 
$$Q'_j = \dfrac{Q_{n+2}}{\widetilde{W}} B'_j - \dfrac{M^2(Q_{n+2}-\widetilde{W} Q_2)}{ \widetilde{W} }B'_{j-2} \textrm{ and } P'_j = (P_0 +\dfrac{P_n}{M^2\widetilde{W}})B'_j - \dfrac{P_n}{\widetilde{W}} B'_{j-2}$$ where $W=B_{n+1}-B_{n-1}$. 
Applying these equations, we obtain 
\begin{equation*}
z'_{2m+2} = z_2 \Leftrightarrow P_0 B'_{m+1} + \dfrac{P_n W}{M^2 \widetilde{W}} B'_m = z_2 \left(\dfrac{W Q_{n+2}}{\widetilde{W}} B'_m + M^2 Q_2 B'_{m-1}\right)
\end{equation*}
Applying $B'_{m-1}=M^{-2}(B'_{m+1}-W B'_m)$, we have
\begin{equation*}
	(P_0-z_2 Q_2) B'_{m+1}+(\dfrac{P_n}{M^2 \widetilde{W}} - \dfrac{z_2 Q_{n+2}}{\widetilde{W}}+z_2 Q_2)WB'_m=0
\end{equation*}
On the other hand, using Equation (\ref{eqn:explicit_PQ}) we obtain $P_n-M^2z_2 Q_{n+2}=(z_1 \widetilde{B_{n-1}}-z_2 \widetilde{B_{n+1}})(M^2 z_2 -z_3)$ and thus
\begin{equation*}
	(z_1-z_2)(z_2-\dfrac{z_3}{M^2}) B'_{m+1}+(z_1-z_2)(z_2-\dfrac{z_3}{M^2})\dfrac{W}{\widetilde{W}}\widetilde{B_{n-1}}B'_m=0
\end{equation*}
We finally obtain Equation (\ref{eqn:last1}) from $W B'_m = \widetilde{W} \widetilde{B'_m}$.
\subsection{Riley-Mednykh polynomial versus the polynomial in Theorem~\ref{thm:solution}} \label{subsec:RMY} 
$$\Lambda=\dfrac{(M^2 z_2-z_3)(z_4-M^2 z_1)}{M^2 z_2 z_3},$$
which is the same $\Lambda$ of the polynomial in Theorem~\ref{thm:solution} is the same as $x$ in the Riley-Mednykh Polynomial.
Note that

\begin{align*}
B'_{2m+1}=B'_{2m+2}-M^{-2} B'_{2m}=S_{m}(z)-M^{-2}S_{m-1}(z),
\end{align*}
\begin{align*}
\widetilde{B'_{2m}}& =\left(B_{2n+1}-B_{2n-1}\right) S_{m-1}(z)\\
&=\left(B_{2n+2}-(M^2+1) B_{2n}+M^2 B_{2n-2}\right) S_{m-1}(z)\\
&=\sqrt{\Lambda}\left(S_n(v)-(M^2+1) S_{n-1}(v)+M^2 S_{n-2}(v)\right) S_{m-1}(z), \qquad \textrm{and}
\end{align*}
\begin{align*}
\widetilde{B_{2n-1}}&=\left(B_{2n}-M^{-2} B_{2n-2}\right) =\dfrac{1}{\sqrt{\Lambda}} \left(S_{n-1}(v)-M^{-2} S_{n-2}(v)\right).
\end{align*}
Hence,
\begin{align*}
&B'_{2m+1}+\widetilde{B'_{2m}} \widetilde{B_{2n-1}}\\
=& S_{m}(z)-M^{-2}S_{m-1}(z)\\
+& \sqrt{\Lambda}\left(S_n(v)-(M^2+1) S_{n-1}(v)+M^2 S_{n-2}(v)\right) S_{m-1}(z) \cdot \dfrac{1}{\sqrt{\Lambda}}\left(S_{n-1}(v)-M^{-2} S_{n-2}(v) \right)\\
=& S_m(z)+\left[-1+x  S_{n-1}(v)\left(S_n(v)+(1-v) S_{n-1}(v)\right)\right]S_{m-1}(z)
\end{align*}
where the second equality comes from $M^{-2}\left(-1+S_n^2(v)-v S_n(v) S_{n-1}(v)\right)=-M^{-2} S_{n-1}^2(v)$, 
$-S_n^2(v)+v S_n(v) S_{n-1}(v)-S_{n-1}^2(v)=-1$, and $v-M^2-M^{-2}=x$.
\section{The complex volume of $J(2n,-2m)$ knot orbifolds} \label{sec:cvolume}
\begin{center}
\begin{longtable}{ccccc}
\caption{The complex volume of $\Oc(J(2n,-2m),r)$}
\label{tab1} \\
\hline
\hline
2n & 2m &r & $\Lambda$ & $i (\textnormal{vol}+i \, \textnormal{cs})(\Oc(J(2n,-2m),r))$  \\
\hline
\hline
 2 & 2 & 4 & 0.50000000000+0.86602540378 i & 4.93480220+0.50747080 i \\
 2 & 2 & 5 & 0.19098300563+0.98159334328 i & 3.94784176+0.93720685 i \\
 2 & 2 & 6 & 1.0000000000 i & 3.28986813+1.22128746 i \\
 2 & 2 & 7 & -0.12348980186+0.99234584135 i & 2.81988697+1.41175465 i \\
 2 & 2 & 8 & -0.20710678119+0.97831834348 i & 2.46740110+1.54386328 i \\
 2 & 2 & 9 & -0.26604444312+0.96396076387 i & 2.19324542+1.63860068 i \\
 2 & 2 & 10 & -0.30901699437+0.95105651630 i & 1.97392088+1.70857095 i \\
 \hline
 \hline
 4 & 2 & 3 & -0.00755235938+0.51311579560 i & -1.72777496+0.65424589 i \\
 4 & 2 & 4 & -0.66235897862+0.56227951206 i & -2.86069760+1.64973788 i \\
 4 & 2 & 5 & -1.14512417943+0.49515849117 i & -3.52261279+2.17889926 i \\
 4 & 2 & 6 & -1.4735614834+0.4447718088 i & -3.98384119+2.47479442 i \\
 4 & 2 & 7 & -1.6951586442+0.4128392229 i & -4.32981047+2.65528044 i \\
 4 & 2 & 8 & -1.8483725477+0.3920489691 i & -4.60000833+2.77325068 i \\
 4 & 2 & 9 & -1.9576333576+0.3778733777 i & -4.81702709+2.85453378 i \\
 4 & 2 & 10 & -2.0378731536+0.3677976891 i & -4.99515284+2.91289047 i \\
 \hline
 \hline
 6 & 2 & 3 & -0.34814666229+0.31569801686 i & -10.75695074+1.13433042 i \\
 6 & 2 & 4 & -1.20132676664+0.23416798146 i & -11.59884416+2.11140104 i \\
 6 & 2 & 5 & -1.7917735176+0.1897754754 i & -12.16988689+2.57280120 i \\
 6 & 2 & 6 & -2.1642555280+0.1684190525 i & -12.59691853+2.82779165 i \\
 6 & 2 & 7 & -2.4067175323+0.1564178761 i & -12.92651554+2.98374073 i \\
 6 & 2 & 8 & -2.5713976386+0.1489407589 i & -13.18751976+3.08602289 i \\
 6 & 2 & 9 & -2.6876701962+0.1439431862 i & -13.39880958+3.15668617 i \\
 6 & 2 & 10 & -2.7725358624+0.1404280473 i & -13.57309955+3.20751918 i \\
 \hline
 \hline
 8 & 2 & 3 & -0.54389004227+0.16911750433 i & -20.0575360+1.3947589 i \\
 8 & 2 & 4 & -1.4978713005+0.1084850975 i & -20.7988735+2.2927467 i \\
 8 & 2 & 5 & -2.1083314541+0.0880656991 i & -21.3519528+2.7201894 i \\
 8 & 2 & 6 & -2.4875315243+0.0785270427 i & -21.7725912+2.9590288 i \\
 8 & 2 & 7 & -2.7331459278+0.0731788731 i & -22.0990803+3.1058946 i \\
 8 & 2 & 8 & -2.8995917551+0.0698442219 i & -22.3583063+3.2025092 i \\
 8 & 2 & 9 & -3.0169658119+0.0676131174 i & -22.5684716+3.2693798 i \\
 8 & 2 & 10 & -3.1025702596+0.0660424620 i & -22.7420003+3.3175427 i \\
 \hline
 \hline
 4 & 4 & 3 & -0.01995829900+0.67074673939 i & -6.57973627+1.76234394 i \\
 4 & 4 & 4 & -0.65138781887+0.75874495678 i & -4.93480220+3.24423449 i \\
 4 & 4 & 5 & -1.09562461427+0.72919649189 i & 3.94784176+3.97745021 i \\
 4 & 4 & 6 & -1.3944865812+0.6970117896 i & 3.28986813+4.37318299 i \\
 4 & 4 & 7 & -1.5970449028+0.6740436568 i & 2.81988697+4.60830212 i \\
 4 & 4 & 8 & -1.7379419910+0.6582482515 i & 2.46740110+4.75900044 i \\
 4 & 4 & 9 & -1.8389076499+0.6471661767 i & 2.19324542+4.86133969 i \\
 4 & 4 & 10 & -1.9133245675+0.6391599851 i & 1.97392088+4.93402245 i \\
 \hline
 \hline
 6 & 4 & 3 & -0.36885266082+0.39541631271 i & -15.6448767+2.3664151 i \\
 6 & 4 & 4 & -1.18085480670+0.33499526222 i & -3.80501685+3.88384993 i \\
 6 & 4 & 5 & -1.7484570879+0.2899432605 i & -4.65930462+4.57539873 i \\
 6 & 4 & 6 & -2.1104185077+0.2675510172 i & -5.25746397+4.94040123 i \\
 6 & 4 & 7 & -2.3472075730+0.2551155981 i & -5.69589799+5.15588538 i \\
 6 & 4 & 8 & -2.5084646758+0.2474927372 i & -6.02968482+5.29371439 i \\
 6 & 4 & 9 & -2.6225041618+0.2424689530 i & -6.29179953+5.38724779 i \\
 6 & 4 & 10 & -2.7058278045+0.2389749437 i & -6.50288760+5.45366136 i \\
  \hline
 \hline
 8 & 4 & 3 & -0.55148024290+0.22035009026 i & -11.80550829+2.69666965 i \\
 8 & 4 & 4 & -1.4793108288+0.1570719654 i & -12.98529702+4.14627902 i \\
 8 & 4 & 5 & -2.0815052970+0.1334359887 i & -13.8048284+4.8015144 i \\
 8 & 4 & 6 & -2.4569465684+0.1227104828 i & -14.3889447+5.1491619 i \\
 8 & 4 & 7 & -2.7004978700+0.1169010753 i & -14.8201199+5.3550962 i \\
 8 & 4 & 8 & -2.8656815894+0.1133764243 i & -15.1496102+5.4870899 i \\
 8 & 4 & 9 & -2.9822244600+0.1110658370 i & -15.4089512+5.5767829 i \\
 8 & 4 & 10 & -3.0672511476+0.1094638695 i & -15.6181360+5.6405274 i \\
 \hline
 \hline
 6 & 6 & 3 & -0.36832830461+0.41796312477 i & -6.57973627+3.00700610 i \\
 6 & 6 & 4 & -1.17336278071+0.37033951782 i & -4.93480220+4.58598208 i \\
 6 & 6 & 5 & -1.7388563928+0.3287642847 i & -3.94784176+5.24850166 i \\
 6 & 6 & 6 & -2.1006443843+0.3072137377 i & -3.28986813+5.58749586 i \\
 6 & 6 & 7 & -2.3375970291+0.2949582783 i & 8.45966092+5.78525139 i \\
 6 & 6 & 8 & -2.4990415162+0.2873325482 i & 7.40220330+5.91106833 i \\
 6 & 6 & 9 & -2.6132384092+0.2822567483 i & 6.57973627+5.99621798 i \\
 6 & 6 & 10 & -2.6966865061+0.2787021856 i & 5.92176264+6.05658540 i \\
 \hline
 \hline
 8 & 6 & 3 & -0.55031840230+0.23422255676 i & -15.9003428+3.3568071 i \\
 8 & 6 & 4 & -1.4724598837+0.1741858085 i & -14.1080969+4.8770938 i \\
 8 & 6 & 5 & -2.0740943804+0.1513440279 i & -5.18916029+5.50547127 i \\
 8 & 6 & 6 & -2.4497492814+0.1408348434 i & -5.83308815+5.82764190 i \\
 8 & 6 & 7 & -2.6935499999+0.1350627376 i & -0.65624423+6.01595120 i \\
 8 & 6 & 8 & -2.8589296734+0.1315233377 i & -1.70960154+6.13591760 i \\
 8 & 6 & 9 & -2.9756183923+0.1291852183 i & -2.52945070+6.21718119 i \\
 8 & 6 & 10 & -3.0607538896+0.1275551112 i & -3.18564118+6.27483020 i \\
 \hline
 \hline
 8 & 8 & 3 & -0.54952912495+0.24145990295 i & -19.7392088+3.7168687 i \\
 8 & 8 & 4 & -1.4711363066+0.1829744778 i & -4.93480220+5.18236078 i \\
 8 & 8 & 5 & -2.0729882919+0.1597713050 i & 3.94784176+5.77639916 i \\
 8 & 8 & 6 & -2.4487718630+0.1489534753 i & 3.28986813+6.08135309 i \\
 8 & 8 & 7 & -2.6926455464+0.1429787861 i & 2.81988697+6.25992917 i \\
 8 & 8 & 8 & -2.8580700183+0.1393045859 i & 2.46740110+6.37384633 i \\
 8 & 8 & 9 & -2.9747882014+0.1368732592 i & 6.57973627+6.45108372 i \\
 8 & 8 & 10 & -3.0599441386+0.1351762880 i & 5.92176264+6.50591258 i \\
\hline
\hline
\end{longtable}
\end{center}

Table~\ref{tab1} gives the complex volume of $\Oc(J(2n,-2m),r)$ for $n$ between $1$ and $4$, $m$ between 1 and 4 with 
$n \geq m$, and $r$ between $3$ and $10$ except the non-hyperbolic orbifold $\Oc(J(2,-2),3)$. One needs to read the Chern-Simons invariant  modulo $2 \pi^2/r$ for $r$ even and $\pi^2/r$ for $r$ odd. One can check that the Chern-Simons invariants of 
$\Oc(J(2m,-2m),r)$ for $m$ between 1 and 4 for $r$ between $3$ and $10$ except the non-hyperbolic orbifold $\Oc(J(2,-2),3)$ are all zero modulo $2 \pi^2/r$ for $r$ even and $\pi^2/r$ for $r$ odd.

We used Mathematica for the calculations. We record here that our data in Table~\ref{tab1} and the orbifold volumes obtained from  SnapPy match up up to existing decimal points. Also, our data in Table~\ref{tab1} and the orbifold Chern-Simons invariants presented in~\cite{HL17} match up up to existing digits when the Chern-Simons invariants in Table~\ref{tab1} are divided by $-2 \pi^2$ and then read by modulo 1/r for $r$ even and are divided by $-\pi^2$, read by modulo 1/r and then divided by $2$ for $r$ odd.
\section{Acknowledgemet}
This work was supported by Basic Science Research Program through the National Research Foundation of Korea (NRF) funded by the Ministry of Education, Science and Technology (No. NRF-2018005847). We would like to thank Cristian Zickert , Jinseok Cho, Seokbeom Yoon (for providing us the $M^2$-deformed solutions for $J(2n,-2m)$), and Darren Long.


\begin{thebibliography}{10}

\bibitem{A1}
N.~V. Abrosimov.
\newblock The {C}hern-{S}imons invariants of cone-manifolds with the
  {W}hitehead link singular set [translation of mr2485364].
\newblock {\em Siberian Adv. Math.}, 18(2):77--85, 2008.

\bibitem{A}
Nikolay~V. Abrosimov.
\newblock On {C}hern-{S}imons invariants of geometric 3-manifolds.
\newblock {\em Sib. \`Elektron. Mat. Izv.}, 3:67--70 (electronic), 2006.

\bibitem{CKK1}
Jinseok Cho, Hyuk Kim, and Seonwha Kim.
\newblock Optimistic limits of kashaev invariants and complex volumes of
  hyperbolic links.
\newblock {\em J. Knot Theory Ramifications}, 23(9), 2014.

\bibitem{CHK}
Daryl Cooper, Craig~D. Hodgson, and Steven~P. Kerckhoff.
\newblock {\em Three-dimensional orbifolds and cone-manifolds}, volume~5 of
  {\em MSJ Memoirs}.
\newblock Mathematical Society of Japan, Tokyo, 2000.
\newblock With a postface by Sadayoshi Kojima.

\bibitem{DMM1}
D.~Derevnin, A.~Mednykh, and M.~Mulazzani.
\newblock Volumes for twist link cone-manifolds.
\newblock {\em Bol. Soc. Mat. Mexicana (3)}, 10(Special Issue):129--145, 2004.

\bibitem{DG1}
Nathan~M. Dunfield and Stavros Garoufalidis.
\newblock Incompressibility criteria for spun-normal surfaces.
\newblock {\em Trans. Amer. Math. Soc.}, 364(11):6109--6137, 2012.

\bibitem{DZ}
Johan~L. Dupont and Christian~K. Zickert.
\newblock A dilogarithmic formula for the {C}heeger-{C}hern-{S}imons class.
\newblock {\em Geom. Topol.}, 10:1347--1372, 2006.

\bibitem{Fox}
Ralph~H. Fox.
\newblock Covering spaces with singularities.
\newblock In {\em A symposium in honor of {S}. {L}efschetz}, pages 243--257.
  Princeton University Press, Princeton, N.J., 1957.

\bibitem{GMT}
Stavros Garoufalidis, Iain Moffatt, and Dylan Thurston.
\newblock Non-peripheral ideal decompositions of alternating knots.
\newblock \url{arXiv:1610.09901}, 2016.
\newblock Preprint.

\bibitem{GZ17}
Matthias Goerner and Christian. Zickert.
\newblock Triangulation independent ptolemy varieties.
\newblock \url{math.GT/1507.03238}, 2017.
\newblock Math. Z. DOI:10.1007/s00209-017-1970-4.

\bibitem{GZ}
Sebastian Goette and Christian~K. Zickert.
\newblock The extended {B}loch group and the {C}heeger-{C}hern-{S}imons class.
\newblock {\em Geom. Topol.}, 11:1623--1635, 2007.

\bibitem{HL}
Ji-Young Ham and Joongul Lee.
\newblock Explicit formulae for {C}hern-{S}imons invariants of the twist knot
  orbifolds and edge polynomials of twist knots.
\newblock {\em Mat. Sb.}, 207(9):144--160, 2016.

\bibitem{HL1}
Ji-Young Ham and Joongul Lee.
\newblock The volume of hyperbolic cone-manifolds of the knot with {C}onway's
  notation {$C(2n,3)$}.
\newblock {\em J. Knot Theory Ramifications}, 25(6):1650030, 9, 2016.

\bibitem{HL2}
Ji-Young Ham and Joongul Lee.
\newblock Explicit formulae for {C}hern--{S}imons invariants of the hyperbolic
  orbifolds of the knot with {C}onway's notation {$C(2n,3)$}.
\newblock {\em Lett. Math. Phys.}, 107(3):427--437, 2017.

\bibitem{HL17}
Ji-Young Ham and Joongul Lee.
\newblock Explicit formulae for {C}hern-{S}imons invariants of the hyperbolic
  ${J}(2n,-2m)$ knot orbifolds.
\newblock \url{www.math.snu.ac.kr/~jyham}, 2017.

\bibitem{HLMR}
Ji-Young Ham, Joongul Lee, Alexander Mednykh, and Aleksei Rasskazov.
\newblock On the volume and {C}hern-{S}imons invariant for 2-bridge knot
  orbifolds.
\newblock {\em J. Knot Theory Ramifications}, 26(12):1750082, 22, 2017.

\bibitem{HLMR1}
Ji-Young Ham, Joongul Lee, Alexander Mednykh, and Aleksey Rasskazov.
\newblock An explicit volume formula for the link $7_3^2 (\alpha, \alpha)$
  cone-manifolds.
\newblock {\em Siberian Electronic Mathematical Reports}, 13:1017--1025, 2016.

\bibitem{HMP}
Ji-Young Ham, Alexander Mednykh, and Vladimir Petrov.
\newblock Trigonometric identities and volumes of the hyperbolic twist knot
  cone-manifolds.
\newblock {\em J. Knot Theory Ramifications}, 23(12):1450064, 16, 2014.

\bibitem{HLM1}
Hugh Hilden, Mar{\'{\i}}a~Teresa Lozano, and Jos{\'e}~Mar{\'{\i}}a
  Montesinos-Amilibia.
\newblock On a remarkable polyhedron geometrizing the figure eight knot cone
  manifolds.
\newblock {\em J. Math. Sci. Univ. Tokyo}, 2(3):501--561, 1995.

\bibitem{HLM3}
Hugh~M. Hilden, Mar{\'{\i}}a~Teresa Lozano, and Jos{\'e}~Mar{\'{\i}}a
  Montesinos-Amilibia.
\newblock On volumes and {C}hern-{S}imons invariants of geometric
  {$3$}-manifolds.
\newblock {\em J. Math. Sci. Univ. Tokyo}, 3(3):723--744, 1996.

\bibitem{HLM2}
Hugh~M. Hilden, Mar{\'{\i}}a~Teresa Lozano, and Jos{\'e}~Mar{\'{\i}}a
  Montesinos-Amilibia.
\newblock Volumes and {C}hern-{S}imons invariants of cyclic coverings over
  rational knots.
\newblock In {\em Topology and {T}eichm\"uller spaces ({K}atinkulta, 1995)},
  pages 31--55. World Sci. Publ., River Edge, NJ, 1996.

\bibitem{HS}
Jim Hoste and Patrick~D. Shanahan.
\newblock A formula for the {A}-polynomial of twist knots.
\newblock {\em J. Knot Theory Ramifications}, 13(2):193--209, 2004.

\bibitem{KKY}
Hyuk Kim, Seonhwa Kim, and Seokbeom Yoon.
\newblock Octahedral developing of knot complement i: pseudo-hyperbolic
  structure.
\newblock \url{arXiv:1612.02928}, 2017.

\bibitem{K88}
Sadayoshi Kojima.
\newblock Isometry transformations of hyperbolic {$3$}-manifolds.
\newblock {\em Topology Appl.}, 29(3):297--307, 1988.

\bibitem{K1}
Sadayoshi Kojima.
\newblock Deformations of hyperbolic {$3$}-cone-manifolds.
\newblock {\em J. Differential Geom.}, 49(3):469--516, 1998.

\bibitem{K2}
Sadayoshi Kojima.
\newblock Hyperbolic {$3$}-manifolds singular along knots.
\newblock {\em Chaos Solitons Fractals}, 9(4-5):765--777, 1998.
\newblock Knot theory and its applications.

\bibitem{MR1}
Alexander Mednykh and Alexey Rasskazov.
\newblock Volumes and degeneration of cone-structures on the figure-eight knot.
\newblock {\em Tokyo J. Math.}, 29(2):445--464, 2006.

\bibitem{MV1}
Alexander Mednykh and Andrei Vesnin.
\newblock On the volume of hyperbolic {W}hitehead link cone-manifolds.
\newblock {\em Sci. Ser. A Math. Sci. (N.S.)}, 8:1--11, 2002.
\newblock Geometry and analysis.

\bibitem{M1}
Alexander~D. Mednykh.
\newblock Trigonometric identities and geometrical inequalities for links and
  knots.
\newblock In {\em Proceedings of the {T}hird {A}sian {M}athematical
  {C}onference, 2000 ({D}iliman)}, pages 352--368. World Sci. Publ., River
  Edge, NJ, 2002.

\bibitem{N2}
Walter~D. Neumann.
\newblock Extended {B}loch group and the {C}heeger-{C}hern-{S}imons class.
\newblock {\em Geom. Topol.}, 8:413--474 (electronic), 2004.

\bibitem{NZ}
Walter~D. Neumann and Don Zagier.
\newblock Volumes of hyperbolic three-manifolds.
\newblock {\em Topology}, 24(3):307--332, 1985.

\bibitem{P2}
Joan Porti.
\newblock Spherical cone structures on 2-bridge knots and links.
\newblock {\em Kobe J. Math.}, 21(1-2):61--70, 2004.

\bibitem{PW}
Joan Porti and Hartmut Weiss.
\newblock Deforming {E}uclidean cone 3-manifolds.
\newblock {\em Geom. Topol.}, 11:1507--1538, 2007.

\bibitem{Sak91}
Tsuyoshi Sakai.
\newblock Geodesic knots in a hyperbolic {$3$}-manifold.
\newblock {\em Kobe J. Math.}, 8(1):81--87, 1991.

\bibitem{SY}
Makoto Sakuma and Yoshiyuki Yokota.
\newblock An application of non-positively curved cubings of alternating links.
\newblock \url{arXiv:1612.06973}, 2016.
\newblock Preprint.

\bibitem{T1}
William Thurston.
\newblock The geometry and topology of 3-manifolds.
\newblock \url{http://library.msri.org/books/gt3m}, 1977/78.
\newblock Lecture Notes, Princeton University.

\bibitem{Tran1}
Anh~T. Tran.
\newblock Twisted {A}lexander polynomials of genus one two-bridge knots.
\newblock \url{arXiv:1506.05035}, 2015.

\bibitem{Tran2}
Anh~T. Tran.
\newblock The {A}-polynomial 2-tuple of twisted whitehead links.
\newblock \url{arXiv:1608.01381}, 2016.

\bibitem{Tran}
Anh~T. Tran.
\newblock Reidemeister torsion and {D}ehn surgery on twist knots.
\newblock {\em Tokyo J. Math.}, 39(2):517--526, 2016.

\bibitem{Weeks}
Jeff Weeks.
\newblock Computation of hyperbolic structures in knot theory.
\newblock In {\em Handbook of knot theory}, pages 461--480. Elsevier B. V.,
  Amsterdam, 2005.

\bibitem{Z1}
Christian~K. Zickert.
\newblock The volume and {C}hern-{S}imons invariant of a representation.
\newblock {\em Duke Math. J.}, 150(3):489--532, 2009.

\bibitem{Z16}
Christian~K. Zickert.
\newblock Ptolemy coordinates, {D}ehn invariant and the {$A$}-polynomial.
\newblock {\em Math. Z.}, 283(1-2):515--537, 2016.

\end{thebibliography}
\end{document}